\documentclass[a4paper]{amsart}
\usepackage{amsmath,amssymb,amsxtra,amsfonts,amsthm,url}

\theoremstyle{plain}
\newtheorem{theorem}{Theorem}
\newtheorem{lemma}{Lemma}

\theoremstyle{definition}
\newtheorem{definition}{Definition}

\newcommand{\B}{\mathbb}

\newcommand{\gf}{\varphi}

\newcommand{\m}[1]{\left(\begin{matrix} #1 \end{matrix}\right)}             %matrix
\makeatletter
  \def\vhrulefill#1{\leavevmode\leaders\hrule\@height#1\hfill \kern\z@}
\makeatother

\begin{document}

\title{The maximal order of hyper-($b$-ary)-expansions}

\author{Michael Coons}
\address{School of Mathematical and Physical Sciences\\
The University of Newcastle\\
Callag\-han, New South Wales\\
Australia}
\email{Michael.Coons@newcastle.edu.au}

\author{Lukas Spiegelhofer}
\address{Institut f\"ur Diskrete Mathematik und Geometrie, Technische Universit\"at Wien, Wien, Austria}
\email{lukas.spiegelhofer@tuwien.ac.at}
\thanks{The research of M.~Coons was supported by Australian Research Council grant DE140100223.
Lukas Spiegelhofer was supported by the Austrian Science Fund (FWF), projects I 1751-N26 and F 5502-N26, the latter of which is a part of the Special Research Program
``Quasi Monte Carlo Methods: Theory and Applications''.}
\date{\today}


%%%%%%%%%%%%%%%%%%%%%%%%%%%%%%%%%%%%%%%%%%%%%%%%
\begin{abstract} Using methods developed by Coons and Tyler, we give a new proof of a recent result of Defant, by determining the maximal order of
the number of hyper-($b$-ary)-expansions of a nonnegative integer $n$
for general integral bases $b\geqslant 2$.
\end{abstract}
%%%%%%%%%%%%%%%%%%%%%%%%%%%%%%%%%%%%%%%%%%%%%%%%

\maketitle

\vspace{-.5cm}
%%%%%%%%%%%%%%%%%%%%%%%%%%%%%%%%%%%%%%%%%%%%%%%%
\section{Introduction}
%%%%%%%%%%%%%%%%%%%%%%%%%%%%%%%%%%%%%%%%%%%%%%%%

If $b\geqslant 2$ is a positive integer and $n\geqslant 0$ is a nonnegative integer, then a {\em hyper-($b$-ary)-expansion of $n$} is an expansion $$n=\sum_{i\geqslant 0} a_ib^i,$$ where $a_i\in\{0,1,\ldots,b\}$ and $a_i=0$ for all but finitely many indices $i$. In contrast, the standard base-$b$ expansion requires $a_i\in\{0,1,\ldots,b-1\}.$ 

For $n\geqslant 1$, we denote by $s_b(n)$ the number of hyper-($b$-ary)-expansions of $n-1$, moreover, we define $s_b(0)=0$.
For example, $\{s_2(n)\}_{n\geqslant 0}$ is the classical diatomic sequence of Stern.
This sequence satisfies $s_2(0)=0$, $s_2(1)=1$, and for $n\geqslant 0$,
$$s_2(2n)=s_2(n) \qquad\mbox{and}\qquad s_2(2n+1)=s_2(n)+s_2(n+1).$$
Similarly, for the general sequence $\{s_b(n)\}_{n\geqslant 0}$, it is straightforward to verify that $s_b(0)=0$, $s_b(1)=1$, and
$$s_b(bn)=s_b(n),\quad s_b(bn+1)=s_b(n)+s_b(n+1)\quad\mbox{and}\quad s_b(bn+i)=s_b(n+1)$$ for $n\geqslant 0$ and $i=2,\ldots,b-1$.

Recently, answering a question of Calkin and Wilf~\cite{CW1998}, Coons and Tyler \cite{CT2014} determined the maximal order of Stern's diatomic sequence. Very recently, Defant~\cite{Dpre} gave an extension of their result to the functions $s_b(n)$. In this short paper, we use a slight variant of the method of Coons and Tyler, to give a new (and much shorter) proof of Defant's result.

\begin{theorem}[Defant]\label{main}
If $b\geqslant 2$ is an integer, then $$\limsup_{m\to\infty}\frac{s_b(m)}{m^{\log_b \varphi}}=\frac{\varphi^{\log_b(b^2-1)}}{\sqrt{5}},$$
where $\varphi=(\sqrt 5+1)/2$ is the golden ratio.
\end{theorem}

%%%%%%%%%%%%%%%%%%%%%%%%%%%%%%%%%%%%%%%%%%%%%%%%
\section{Preliminaries}
%%%%%%%%%%%%%%%%%%%%%%%%%%%%%%%%%%%%%%%%%%%%%%%%
Let $\{F_k\}_{k\geqslant 0}=0,1,1,2,\ldots$ be the sequence of Fibonacci numbers.

\begin{lemma}[Defant] Let $k\geqslant 2$ and $n$ be positive integers. Then $$\max_{b^{k-2}\leqslant n<b^{k-1}}s_b(n)=F_{k}.$$
Moreover, if $a_k$ denotes the smallest $n$ in the interval $[b^{k-2},b^{k-1})$ for which this maximum is attained, then
\begin{equation}\label{eqn:ak_explicit}
a_k=\frac{b^{k}-1}{b^2-1}+\left(\frac{1-(-1)^{k}}{2}\right)\frac b{b+1}.\end{equation}
%Stern: a_2=1, a_3=3, a_4=5, a_5=11 etc.
%1 12 1323 14352534
\end{lemma}
Thus by definition $$s_b(a_k)=F_k.$$
We note that $a_k$ has the base $b$-expansion $((10)^{\ell-1} 1)_b$ for $k=2\ell$
and $((10)^{\ell-1} 11)_b$ for $k=2\ell+1$,
therefore this Lemma follows from a result of Defant \cite[Proposition 2.1]{Dpre},
however, this lemma can also be proven directly from the corresponding statement for Stern's sequence (see Lehmer \cite{L1929} and Lind \cite{L1969}).
To do so, one need only notice that the Stern sequence is a subsequence of $\{s_b(n)\}_{n\geqslant 0}$ and that the maximal values occur in this subsequence first.
Explicitly, if we define $\psi_b:\B{N}\to\B{N}$ by $$\psi_b:(\varepsilon_d\ldots\varepsilon_0)_2\mapsto(\varepsilon_d\ldots\varepsilon_0)_b,$$ then setting $$A_0=\m{1&0\\1&1},\quad A_1=\m{1&1\\0&1},\quad\mbox{and}\quad A_2=\cdots=A_{b-1}=\m{0&1\\0&1},$$ we easily see that for $n=(\varepsilon_d\ldots\varepsilon_0)_b$, we have
\begin{equation}\label{eqn:b_regular}
s_b(n)=\m{1&0}A_{\varepsilon_0}\cdots A_{\varepsilon_d}\m{0\\1}.
\end{equation}
From this identity it follows at once that
\begin{equation}\label{eqn:stern_as_subsequence}
s_b(\psi_b(n))=s_2(n).
\end{equation}
Moreover, the matrices $A_i$ are nonnegative and $A_i$ is bounded componentwise by $A_1$ for $i\geqslant 2$. Replacing any such matrix $A_i$ with $i\geqslant 2$ by $A_1$ therefore does not decrease the value of the product in~\eqref{eqn:b_regular}. This proves that
\[
s_b((\varepsilon_d\ldots\varepsilon_0)_b)\leqslant
s_2((\tilde\varepsilon_d\ldots\tilde\varepsilon_0)_2),\]
where $\tilde\varepsilon_i=\min(\varepsilon_i,1)$.
Therefore numbers with only zeros and ones in their base-$b$ expansions dominate the sequence $s_b(n)$, that is, for every $n$ there is an $m\leqslant n$ with only zeros and ones in its base-$b$ expansion such that $s_b(m)\geqslant s_b(n).$
\begin{definition}\label{defak}
For $k\geqslant 0$ set \[  \tilde a_k=\frac{b^k-1}{b^2-1},  \] and
let $h(x):\B{R}_{\geqslant 0}\to\B{R}_{\geqslant 0}$ denote the piecewise linear function connecting the sequence of points $\{(\tilde a_k,F_k)\}_{k\geqslant 0}=(0,0), (1/(b+1),1), (1,1),\ldots$.
\end{definition}
In the following lemma, we collect some necessary properties of $h(x)$.
%{{{
\begin{lemma}\label{lem:h_properties}
The function $h(x)$ is continuous in $[0,\infty)$, monotonically increasing, and differentiable in the intervals $(\tilde a_k,\tilde a_{k+1})$.
Moreover, for all $x\geqslant 0$, we have 
\begin{equation}\label{eqn:h_recurrence}
h(x)+h(bx+1/(b+1))=h(b^2x+1).
\end{equation}
For $x\in[\tilde a_k,\tilde a_{k+1}]$ we have
\begin{equation}\label{eqn:h_explicit}
h(x)
=F_{k-1}\frac{b+1}{b^k}x-F_{k-1}\frac{b^k-1}{b^k(b-1)}+F_k.
\end{equation}
Moreover the sequence $\{\gamma_k\}_{k\geqslant 2}=(b+1)/b^2,(b+1)/b^3,\ldots$ of slopes of the line segments is nonincreasing.
In particular, the function $h$ is concave in $\B{R}_{\geqslant 1}$.
\end{lemma}
%}}}
\begin{proof}
The statements in the first sentence follow directly from the definition.

To establish the validity of Equation~\eqref{eqn:h_recurrence}, it is enough to realise that on the interval $[\tilde a_k,\tilde a_{k+1}]$, both sides are linear functions, which coincide at the endpoints, since $h(\tilde a_k)+h(\tilde a_{k+1})=F_k+F_{k+1}=F_{k+2}=h(\tilde a_{k+2})$.

The proof of~\eqref{eqn:h_explicit} is by inserting the definition of $\tilde a_k$
into the equation $h(x)=(x-\tilde a_k)(F_{k+1}-F_k)/(\tilde a_{k+1}-\tilde a_k)+F_k$.
Finally, it is an easy consequence of Binet's formula that the sequence of slopes is nonincreasing.
\end{proof}
The piecewise linear function $h(x)$ thus satisfies a recurrence relation resembling that of $s_b(n)$. Using this recurrence, we show that $h(n)$ is, in fact, an upper bound for $s_q(n)$, which is the main tool of our proof.

\begin{lemma}\label{lem:bounded_by_polygon} Assume that $b\geqslant 2$ is an integer.
For all $m\geqslant 0$, we have
\begin{equation}\label{eqn:bounded_by_polygon}
s_b(m)\leqslant h(m).
\end{equation}
Moreover,
\begin{equation}\label{eqn:upper_limit}
\limsup_{m\to\infty}\frac{s_b(m)}{h(m)}=1.
\end{equation}
\end{lemma}

\begin{proof}
The limit result~\eqref{eqn:upper_limit} follows from the first statement, which yields the inequality ``$\leqslant$'', together with the equality $s_b(a_{2k})=h(a_{2k})$ that holds for all $k\geqslant 0$.
It is sufficient to prove the assertion~\eqref{eqn:bounded_by_polygon} for such $m$ whose $b$-ary expansion consists only of zeros and ones, as for each $n$ there is such an $m$ with $m\leqslant n$ and $s_b(n)\leqslant s_b(m)$; see the comments just before Definition \ref{defak} for details.

We prove the lemma using induction, however, we have to strengthen the induction hypothesis by the additional property that
\begin{equation}\label{eqn:strengthening}
s_b(m)\leqslant h(m-b/(b+1)) \quad\mbox{for}\quad m\equiv b+1\bmod b^2.
\end{equation}
(Remark: in fact we are working with a strictly weaker bound than Coons and Tyler~\cite{CT2014} did for the case $b=2$. They defined the function $h$ by just connecting the maxima at the positions $a_k$.
Our function $h$ hits only those maxima $(a_k,F_k)$ where $2\mid k$, and the others lie to the right of $h$, shifted by $b/(b+1)$.
Therefore our induction hypothesis has to be stronger, but as an exchange for this additional difficulty the tedious proof of Lemma~2.1 in the cited paper can be dispensed with, since our bound $h$ is easier to work with.)

To get the proof started, we first verify the statement for the sixteen integers $m$ given by the $b$-ary expansions $(\varepsilon_3\varepsilon_2\varepsilon_1\varepsilon_0)_b$, where $\varepsilon_i\in\{0,1\}$. The case that $\varepsilon_i=0$ for all $i\in\{0,1,2,3\}$ is trivial.
Concerning the other fifteen indices, by monotonicity of $h$ we only have to consider those positions where a new maximum is attained: we have $s_b((1)_b)=1$, $s_b((11)_b)=2$, $s_b((101)_b)=3$, $s_b((1001)_b)=4$, $s_b((1011)_b)=5$ and these values do not appear before.
Clearly $s_b(1)=1=h(\tilde a_2)=h(1)$ and $s_b(b^2+1)=3=h(\tilde a_4)=h(b^2+1)$. Moreover,
$s_b(b+1)=2=h(\tilde a_3)=h(b+1-b/(b+1))$
and $s_b(b^3+b+1)=5=h(\tilde a_5)=h(b^3+b+1-b/(b+1))$, so that~\eqref{eqn:strengthening} is satisfied for these latter two positions.
Finally, we have $s_b(b^3+1)=4$ and by~\eqref{eqn:h_explicit} it is not difficult to show that $h(b^3+1)=5-2/b^2\geqslant 4$.

For the induction step, let $k\geqslant 5$ and assume that~\eqref{eqn:bounded_by_polygon} and~\eqref{eqn:strengthening} hold for all $m<b^{k-1}$ having only zeros and ones as $b$-ary digits. Let $b^{k-1}\leqslant m<b^k$, where $m$ satisfies the same digit restriction.
We distinguish between six cases.
%{{{the case "0 mod 2"
If $m=bj$, where $j\geqslant 1$, then $$s_b(bj)=s_b(j)\leqslant h(j)\leqslant h(bj),$$ where the last inequality is true by the monotonicity of $h(x)$. 
%}}}
The case for $m=bn+1$ for some $n$ is more involved, and it is this case for which we need the special form of the function $h$.
%{{{the case "1 mod 8"
We start with the case $m\equiv 1\bmod b^3$, that is, $m=b^3j+1$, where $j\geqslant b$ since $m\geqslant b^4$.
Using the recurrence for $s_b$, the induction hypothesis, monotonicity of $h$ and the identity~\eqref{eqn:h_recurrence}, in this order, we obtain
\begin{align*}
s_b(b^3j+1)&=s_b(b^2j)+s_b(b^2j+1)
\\&=s_b(j)+s_b(bj)+s_b(bj+1)
\\&=2s_b(j)+s_b(bj+1)
\\&\leqslant 2h(j)+h(bj+1)
\\&\leqslant h(j)+h(j+1/(b+1))+h(bj+1)
\\&=h(j)+h(b^2j+b+1/(b+1)).
\end{align*}
The first case $m\equiv 1\bmod b^3$ is finished as soon as we prove that this last expression is bounded from above by 
\[h(bj)+h(b^2j+1/(b+1))=h(b^3j+1).\]
This estimate follows from the concavity of $h$ in $\B{R}_{\geqslant 2}$, comparing the arguments of the function $h$: we have $j\geqslant b\geqslant 2$, and therefore $bj-j\geqslant b=b^2j+b+1/(b+1)-(b^2+1/(b+1))$, which yields the statement.
%}}}

%{{{the case "3 mod 16"
By a very similar argument we treat the case $m=b^4j+b+1$, however we consider the case $j=1$ separately. Using~\eqref{eqn:stern_as_subsequence} and~\eqref{eqn:h_explicit}, we get in this case
\[s_b(b^4+b+1)=s_2(19)=7\leqslant 8-3/b^2=h(b^4+b+1/(b+1)),\] as required by~\eqref{eqn:strengthening}.
Assume now that $j\geqslant 2$.
We have (note that the argument leading to the second row works equally well for $b=2$)
\begin{align*}
s_b(b^4j+b+1)&=s_b(b^3j+1)+s_b(b^3j+2)
\\&=s_b(b^3j+1)+s_b(b^2j+1)
\\&=s_b(j)+2s_b(b^2j+1)
\\&\leqslant h(j)+2h(b^2j+1)
\end{align*}
and an analogous argument as before, using the concavity of $h$, shows that this can be estimated by
\begin{align*}
h(bj)+2h(b^2j+1/(b+1))&=h(b^2j+1/(b+1)+h(b^3j+1)
\\&=h(b^4j+b+1-b/(b+1)).
\end{align*}
%}}}

The remaining three cases are simpler as we do not need the concavity of $h$ for them.
%{{{the case "11 mod 16"
We have
\begin{align*}
s_b(b^4j+b^3+b+1)&=s_b(b^3j+b^2+1)+s_b(b^3j+b^2+2)
\\&=s_b(b^3j+b^2+1)+s_b(b^2j+b+1)
\\&\leqslant h(b^3j+b^2+1)+h(b^2j+b+1/(b+1))
\\&=h(b^4j+b^3+b+1-b/(b+1)),
\end{align*}
%}}}
%{{{the case "5 mod 8"
\begin{align*}
s_b(b^3j+b^2+1)&=s_b(bj+1)+s_b(b^2j+b+1)
\\&\leqslant h(bj+1)+h(b^2j+b+1/(b+1))
\\&=h(b^3j+b^2+1)
\end{align*}
%}}}
%{{{the last case, "7 mod 8".
and
\begin{align*}
s_b(b^3j+b^2+b+1)&=s_b(b^2j+b+1)+s_b(b^2j+b+2)
\\&=s_b(b^2j+b+1)+s_b(bj+2)
\\&=s_b(b^2j+b+1)+s_b(j+1)
\\&\leqslant h(b^2j+b+1/(b+1))+h(j+1)
\\&\leqslant h(b^2j+b+1/(b+1))+h(bj+1)
\\&=h(b^3j+b^2+1)
\\&\leqslant h(b^3j+b^2+b+1-b/(b+1)),
\end{align*}
%}}}
which settles the cases $m\equiv b^3+b+1\bmod b^4$, $m\equiv b^2+1\bmod b^3$ and $m\equiv b^2+b+1\bmod b^3$ respectively.
Note that all of these calculations are also valid for $b=2$.

The residue classes considered cover the set of nonnegative integers with $b$-ary expansions containing only zeros and ones, and therefore $$s_b(m)\leqslant h(m)$$ for all $m\geqslant 0$.
\end{proof}

\begin{lemma}\label{lem:continuous_upper_bound}
If
\[H(x)=\frac{(b^2-1)^{\log_b \gf}}{\sqrt{5}}x^{\log_b \varphi},\]
then
\[  \limsup_{x\rightarrow \infty}\,(h-H)(x)=0,  \]
 and specifically,
\[\lim_{k\rightarrow \infty}|(h-H)(a_k)|=0.\]
\end{lemma}

\begin{proof}
Let $\tilde h(x)$ be the piecewise linear function satisfying
$\tilde h(\tilde a_k)=\varphi^k/\sqrt{5}$ for $k\geqslant 0$, and set
\[\tilde H(x)=\frac{((b^2-1)x+1)^{\log_b \varphi}}{\sqrt{5}}=H(x+1/(b^2-1)).\]
Then $H(x)$ is a concave function that passes through the points $\bigl(\tilde a_k,\varphi^k/\sqrt{5}\bigr)$. It follows that $\tilde h(x)\leqslant \tilde H(x)$ for $x\geqslant 0$, where we have equality when $x=\tilde a_k$.
Moreover, if $\tilde a_k\leqslant x\leqslant \tilde a_{k+1}$, by Binet's formula $|\tilde h(x)-h(x)|\leqslant \varphi^{-k}/\sqrt{5}$.
Furthermore, the mean value theorem gives $|\tilde H(x)-H(x)|\leqslant C x^{-\eta}$ for some $C$ and some $\eta>0$.
Combining these three estimates yields the first statement.

We easily get $\lim_{k\rightarrow \infty}|(h-H)(\tilde a_k)|=0$ by a similar argument. Combining this with $|a_k-\tilde a_k|\leqslant 1$, $|h(x+t)-h(x)|\rightarrow 0$, and $|H(x+t)-H(x)|\rightarrow 0$ for $x\rightarrow\infty$ gives the last statement of the lemma.
\end{proof}

%%%%%%%%%%%%%%%%%%%%%%%%%%%%%%%%%%%%%%%%%%%%%%%%
\section{Proof of the main result}
%%%%%%%%%%%%%%%%%%%%%%%%%%%%%%%%%%%%%%%%%%%%%%%%

\begin{proof}[Proof of Theorem \ref{main}]
For brevity, we write $c_b=\varphi^{\log_b(b^2-1)}/\sqrt{5}$.
By Lemma~\ref{lem:continuous_upper_bound}, noting that $h(x)$ increases to infinity, and applying Lemma~\ref{lem:bounded_by_polygon}, we get
\[
\limsup_{m\rightarrow\infty}
\frac{s_b(m)}{c_bm^{\log_b\gf}}
=\limsup_{m\rightarrow\infty}
\frac{s_b(m)}{H(m)}
\leqslant \limsup_{m\rightarrow\infty}
\frac{s_b(m)}{h(m)}\leqslant 1.
\]
Similarly, using the second part of Lemma~\ref{lem:continuous_upper_bound} and again applying Lemma~\ref{lem:bounded_by_polygon}, we have
\[
\limsup_{m\rightarrow\infty}
\frac{s_b(m)}{H(m)}
\geqslant
\limsup_{k\rightarrow\infty}
\frac{s_b(a_k)}{H(a_k)}
=\limsup_{k\rightarrow\infty}
\frac{s_b(a_k)}{h(a_k)}=1,
\]
which proves the theorem.
\end{proof}

%%%%%%%%%%%%%%%%%%%%%%%%%%%%%%%%%%%%%%%%%%%%%%%%
\section{Concluding remarks}
%%%%%%%%%%%%%%%%%%%%%%%%%%%%%%%%%%%%%%%%%%%%%%%%

In this paper, we gave a short proof determining the maximal order of the number $s_b(n)$ of hyper-($b$-ary)-expansions of a nonnegative integer $n-1$ for general integral bases $b\geqslant 2$. Our proof was based on considering the finite number of specific recurrences satisfied by $s_b(n)$ over arithmetic progressions $bn+i$, and constructing a piecewise linear function approximating those recurrences. 

Functions satisfying recurrences like those satisfied by $s_b(n)$ are plentiful in the literature, and often given special attention by number theorists and theoretical computer scientists; see Allouche and Shallit's work on $b$-regular sequences \cite{AS1992, AS2003} for a general treatment and several specific examples. The recurrences satisfied by $b$-regular sequences lead to a matrix classification, which has proven to be very useful; here the usefulness comes by recognising that the Stern sequence $s_2(n)$ is the dominating subsequence of $s_b(n)$ (for each $b\geqslant 2$), and that the growth properties of the Stern sequence determine the growth properties of $s_b(n)$. 

We end by pointing the interested reader to the extremely useful result of Allouche and Shallit \cite[Theorem 2.2]{AS1992}, which formalises this relationship.

\begin{theorem}[Allouche and Shallit] The integer sequence $\{f(n)\}_{n\geqslant 0}$ is $b$-regular if and only if there exist positive integers $m$ and $d$, matrices ${\bf A}_0,\ldots,{\bf A}_{b-1}\in \B{Z}^{d\times d}$, and vectors ${\bf v},{\bf w}\in \B{Z}^d$ such that $$f(n)={\bf w}^T {\bf A}_{i_0}\cdots{\bf A}_{i_s} {\bf v},$$ where $n=({i_s}\cdots {i_0})_b.$
\end{theorem}

%%%%%%%%%%%%%%%%%%%%%%%%%%%%%%%%%%%%%%%%%%%%%%%%%
\bibliographystyle{amsplain}
\providecommand{\bysame}{\leavevmode\hbox to3em{\hrulefill}\thinspace}
\providecommand{\MR}{\relax\ifhmode\unskip\space\fi MR }
% \MRhref is called by the amsart/book/proc definition of \MR.
\providecommand{\MRhref}[2]{%
  \href{http://www.ams.org/mathscinet-getitem?mr=#1}{#2}
}
\providecommand{\href}[2]{#2}

%%%%%%%%%%%%%%%%%%%%%%%%%%%%%%%%%%%%%%%%%%%%%%%%%

\end{document}